\documentclass[11pt]{amsproc}
\usepackage{amsmath,amssymb,amsthm,framed,cancel,tikz}
\usetikzlibrary{arrows, decorations.markings}
\tikzstyle{vecArrow} = [thick, decoration={markings,mark=at position
   1 with {\arrow[semithick]{open triangle 60}}},
   double distance=1.4pt, shorten >= 5.5pt,
   preaction = {decorate},
   postaction = {draw,line width=1.4pt, white,shorten >= 4.5pt}]
\tikzstyle{innerWhite} = [semithick, white,line width=1.4pt, shorten >= 4.5pt]

\newtheorem{theorem}{Theorem}[section]
\newtheorem{lemma}[theorem]{Lemma}

\newtheorem{proposition}[theorem]{Proposition}
\newtheorem{corollary}[theorem]{Corollary}
\newtheorem{remark}[theorem]{Remark}

\theoremstyle{definition}
\newtheorem{definition}[theorem]{Definition}

\newtheorem*{question}{Question}
\theoremstyle{remark} 
\theoremstyle{remark}\newtheorem*{case}{Case} 
\theoremstyle{remark}
\newcommand{\N}{\mathbb{N}}
\newcommand{\R}{\mathbb{R}}

\numberwithin{equation}{section}

\begin{document}
\title[On operators on $C_0(\alpha\times L)$]{On operators on $C_0(\alpha\times L)$ under the Ostaszewski's $\clubsuit$-principle}
\date{ }
\author{Leandro Candido}
\address{Universidade Federal de S\~ao Paulo - UNIFESP. Instituto de Ci\^encia e Tecnologia. Departamento de Matem\'atica. S\~ao Jos\'e dos Campos - SP, Brasil}
\email{\texttt{leandro.candido@unifesp.br}}
\thanks{ The author was supported by Funda\c c\~ao de Amparo \`a Pesquisa do Estado de S\~ao Paulo - FAPESP No. 2016/25574-8 }




\begin{abstract} 
For an exotic locally compact Hausdorff space $L$, cons\-truc\-ted under the assumption of the Ostaszewski's $\clubsuit$-principle, 
and a countable ordinal space $\alpha$, we prove that all operators defined on $C_0(\alpha\times L)$ are as simple as possible. We also investigate the geometry of such space $C_0(\alpha\times L)$ and we classify up to isomorphisms all its complemented subspaces.
\end{abstract}

\maketitle

\section{Basic terminology and notation}

The letters $K$ and $L$ indexed or not, unless explicitly stated otherwise,  will always stand for a Hausdorff compactum and a locally compact Hausdorff space respectively. We write $K=L\cup \{\infty\}$ if $K$ is the Aleksandrov one-point compactification of $L$. $C(K)$ is the Banach space of all continuous functions $f: K \to \mathbb{R}$, with the norm: $\|f\| = \sup_{x\in K} \|f(x)\|$. If $K=L\cup \{\infty\}$, $C_0(L)$ is the subspace of $C(K)$ consisting of all functions $f$ which vanish at infinity, that is, $f(\infty)=0$.

For spaces $K_1=L_1\cup \{\infty_1\}$ and $K_2=L_2\cup \{\infty_2\}$, their product $K_1\times K_2$ is endowed with the usual product topology. If $\partial(K_1\times K_2)=(\{\infty_1\}\times K_2)\cup (K_1 \times \{\infty_2\})$, then $C_0(L_1\times L_2)$ can be  isometrically identified with the subspace of $C(K_1\times K_2)$ of all functions that vanish at $\partial(K_1\times K_2)$. The dual space $C(K_1\times K_2)^*$ is identified, via Riesz representation theorem, with $M(K_1\times K_2)$, the space of all signed Radon measures on $K_1\times K_2$ of finite variation endowed with the variation norm. For every measure $\mu$, $|\mu|$ denotes its variation. It is standard to check that the dual space $C_0(L_1\times L_2)^*$ can be identified with the subspace of $M(K_1\times K_2)$ of all measures $\mu$ such that $|\mu|(\partial(K_1\times K_2))=0$.
 
For any given $f:K_1\to \mathbb{R}$ and $g:K_2\to \mathbb{R}$, $f\otimes g:K_1\times K_2\to \mathbb{R}$ is the function given by $f\otimes g(x,y)= f(x)\cdot g(y)$. For subspaces $X\subset C_0(L_1)$ and $Y\subset C_0(L_2)$, $X\otimes_{\varepsilon}Y$ is the linear span of the set $\{f\otimes g: f \in X, g \in Y\}$ endowed with the norm from $C_0(L_1\times L_2)$. Its closure in $C_0(L_1\times L_2)$, denoted by $X\widehat{\otimes}_{\varepsilon}Y$, is the injective tensor product of $X$ and $Y$. In particular, see \cite[\S 20]{Se}, $C_0(L_1)\widehat{\otimes}_{\varepsilon}C_0(L_1)=C_0(L_1\times L_2)$.
%

Lower case Greek letters will denote ordinal numbers, in particular, $\omega$ is the least infinite ordinal 
and $\omega_1$ is the least uncountable ordinal. By abuse of notation, we will also denote the least infinite cardinal by $\omega$ and the least uncountable cardinal as $\omega_1$. An ordinal $\alpha$, as a topological space, is always endowed with its usual order topology. But $c_0(\omega_1)$ will denote the space $C_0(\varGamma)$ where $\varGamma$ is a discrete space of cardinality $|\varGamma|=\omega_1$. The ordinal spaces $\alpha$ and $\alpha+1$ can also be denoted by $[0,\alpha)$ and $[0,\alpha]$ respectively.   

For a topological space $K$, $K^{(1)}$ is the set of all its non-isolated points. For an ordinal number $\rho$, the $\rho$-Cantor-Bendixson derivative $K^{(\rho)}$ is defined as follows:  $K^{(0)}=K$ and 
\begin{displaymath}
K^{(\rho)}= \left\{
\begin{array}{ll}
\left(K^{(\gamma)}\right)^{(1)} & \text{ if }\rho=\gamma+1;\\
\bigcap_{\gamma<\rho}K^{(\gamma)} & \text{ if }\rho\text{ is a limit ordinal}.
\end{array} \right.
\end{displaymath}

We recall that a topological space $K$ is scattered if every nonempty subset $A \subset K$ has an isolated point in $A$. In this case, there will be an ordinal number $\rho$ such that $K^{(\rho)}=\emptyset$ and we call the least such ordinal as the \emph{height} of $K$.

If $K$ is a scattered compactum, then $M(K)$ is isometrically isomorphic to $\ell_1(K)$ (see \cite{Rud}). In this case, for each $\mu \in M(K)$, we will denote by $\mathrm{supp}(\mu)$ the countable set $\{z\in K:\mu(\{z\})\neq 0\}$ (the support of a function $f:L\to \R$, that is,  the set $\{x\in L:f(x)\neq 0\}$, will also be denoted by $\mathrm{supp}(f)$).

For all Banach spaces $X$ and $Y$, we say that $T:X\to Y$ is an operator if it is linear and continuous. If $T$ is an isomorphism, we say that $X$ is isomorphic to $Y$ and write $X\sim Y$. If $T$ is an isometric isomorphism, we say that $X$ is isometrically isomorphic to $Y$ and write $X\cong Y$. If an operator $T$ is such that $T^2=T\circ T=T$ we say that $T$ is a projection. If $T$ is a projection with image $Y$, we say that $Y$ is complemented in $X$.
 

\section{Introduction}

In 1976, see \cite{Lind}, J. Lindenstrauss posed the following
\begin{question}
Does there exist an infinite dimensional Banach space $X$ so that each operator $T:X\to X$ is of the form
$T=aI+S$, where $a$ is a scalar, $I:X\to X$ is the identity operator and $S:X\to X$ is a compact operator?
\end{question}

Lindenstrauss observed that if the previous question had an affirmative answer, the expected space $X$ would present a particularly exotic geometry. 
More specifically, for any projection $P:X\to X$ with image $Y$ there is a scalar $a$ and a compact operator $S$ such that $P=aI+S$. Therefore
\[P^2=(aI+S)\circ (aI+S)=a^2I+2aS+S^2=aI+S=P\] 
and we may deduce that $(a^2-a)I=S'$ where $S'=S-2aS-S^2$ is compact operator. Since the identity $I$ is not compact,  
$a=1$ or $a=0$. If $a=0$, the projection $P$ is itself compact. Therefore $Y$ is finite dimensional. On the other hand, if 
$a=1$, then the complementary projection $P-I=S$ is a compact operator. Consequently $Y$ is a subspace of finite codimension. 
We conclude that if $X=Y\oplus Z$, then either $Y$ of $Z$ has finite dimension. Banach spaces with this property are called indecomposable.

The question of Lindenstrauss was completely solved in the affirmative in 2011 by S. A. Argyros and R. G. Haydon in \cite{ArgyrosHaydon}. It had, nevertheless, inspired several other fascinating lines of research is the past decades. From works of S. Shelah \cite{Shelah} to the celebrated work of  W. T. Gowers and B. Maurey \cite{GowMor} where the famous problem of unconditional basic sequence was solved: there is a Banach space possessing a Schauder bases without unconditional basic sequence. W. T. Gowers and B. Maurey constructed in \cite{GowMor} a Banach space $X$ such that every operator $T:X\to X$ is of the form $T = aI + S$ where $a$ is a scalar and $S$ is a strictly singular operator. We recall that an operator $S:X\to X$ is said to be strictly singular if there is no infinite dimensional subspace $Y$ of $X$ such that the restriction of $S$ to $Y$ is an isomorphism onto its image. Furthermore, the space of Gowers and Maurey is hereditarily indecomposable (H.I.), that is, every infinite dimensional subspace is indecomposable.

The question of Lindestrauss can be also considered in the context of $C(K)$ spaces. In 2004, see \cite{Koszmider1}, under the assumption of the continuum hypothesis,  P. Koszmider constructed a connected Hausdorff compactum $K$ such that any linear bounded operator $T : C(K)\to C(K)$ is of the form $gI + S$  where $g\in C(K)$, $I$ denotes the identity operator and $S$ is a weakly compact linear operator on $C(K)$ or equivalently (in $C(K)$ spaces) strictly singular. Later, G. Plebanek obtained a $C(K)$ space with the same properties but his construction does not depend on the continuum hypothesis, see \cite{Pleb}.

Similar questions are also interesting for Asplund spaces of the form $C(K)$, that is, when $K$ is a scattered compactum. It is important observe that if $K$ is an infinite scattered compactum, them $C(K)$ has a complemented copy of $c_0$. Such copy induce many 
operators that are not of the form $T=a I + S$ where $S:C(K)\to C(K)$ is a compact operator, see \cite{Kosz1}.

In 2005, see \cite{Koszmider2}, by assuming either the continuum hypothesis or its negation and the Martin's Axiom, P. Koszmider
obtained an example of a scattered Hausdorff compactum $K$ such that every operator $T : C(K)\to C(K)$ is of the form 
$T = a I + S$ where $a$ is a fixed real number and $S$ has its image included in a copy of $c_0$. Moreover, if $C(K)\sim A\oplus B$ where $A$ and $B$ are infinite-dimensional Banach spaces then $A\sim C(K)$ and $B\sim c_0$, or vice versa.

It is possible to obtain different examples with related properties under different extra set-theoretic assumptions. In 2011, see \cite{KZ}, under the assumption of Ostaszewski's $\clubsuit$-principle \cite{Osta}, P. Koszmider and P. Zieli\'nski presented a 
scattered compactum $K$ such that every operator $T : C(K)\to C(K)$ is of the form $T = a I + S$ where $a \in \mathbb{R}$ is a constant, $S$ has separable image included in a complemented subspace of $c_0$ or $c_0(\omega^{\omega})$. 
Moreover, if $C(K)\sim A\oplus B$ where $A$ and B are infinite-dimensional Banach spaces then $A\sim C(K)$ and $B\sim c_0$ or $B\sim C_0(\omega^\omega)$, or vice versa.

In this paper we investigate the same type of questions above for spaces of the form $C_0(\alpha\times L)$. More precisely, 
inspired in a construction from \cite{CK} we isolate a property for a locally compact Haudorsff space $L$ which will imply, for each countable ordinal space $\alpha$, that every operator on $C_0(\alpha\times L)$ is as simple as possible. 

We will see (Proposition \ref{oper6}) that each operator $R:C_0(\alpha)\to C_0(\alpha)$ induces an operator 
$R_L:C_0(\alpha\times L)\to C_0(\alpha\times L)$ through the following formula
\[R_L(f)(x,y)=R(f\restriction_{\alpha\times \{y\}})(x).\] 
Where for each $y\in L$, $f\restriction_{\alpha\times \{y\}}$ denotes the function $x\mapsto f(x,y)\in C_0(\alpha)$. 

Therefore, there are always on $C_0(\alpha\times L)$ operators of the form 
$T=R_L+S$, where $R_L$ is induced by an operator $R:C_0(\alpha)\to C_0(\alpha)$ and $S:C_0(\alpha\times L)\to C_0(\alpha\times L)$ is an operator of separable image.

We must, however, observe that it may exist operators that are not of previous form, see Remark \ref{geom2222}. 

In our main result, under the assumption of Ostaszewski's 
$\clubsuit$-principle, we will be present a scattered locally compact space $L$ such that all operators on $C_0(\alpha\times L)$  are of the form $T=R_L+S$. More specifically,

\begin{theorem}[$\clubsuit$]\label{main1}There is a non-metrizable scattered locally compact Hausdorff space $L$ such that, for every countable ordinal spaces $\alpha$ and $\beta$, for every operator $T:C_0(\alpha\times L)\to C_0(\beta\times L)$, there is a unique operator $R:C_0(\alpha)\to C_0(\beta)$ and a unique operator $S:C_0(\alpha\times L)\to C_0(\beta\times L)$  with separable image  such that $T=R_L+S$.
\end{theorem}

If $K=L\cup\{\infty\}$ where $L$ is the space from Theorem \ref{main1}, we have the following

\begin{corollary}\label{main2}
For any metric compacta $M_1$ and $M_2$, 
\[C(M_1\times K)\sim C(M_2\times K)\iff C(M_1)\sim C(M_2).\]
\end{corollary}

In particular, for all $n,m<\omega$, $C(K)^n\sim C(K)^m$ if and only if $n=m$.

As a consequence of Theorem \ref{main1}, the geometry of such space $C_0(\alpha\times L)$ is completely revealed. It is possible to classify up to isomorphism all its complemented subspaces. Namely,

\begin{theorem}[$\clubsuit$]\label{main3}There is a non-metrizable locally compact scattered Hausdorff space $L$ such that, for any
countable ordinal $\alpha$,  $X$ is a complemented subspace of $C_0(\alpha\times L)$ if and only if 
$X\sim (A \widehat{\otimes}_{\varepsilon} C_0(L)) \oplus B$ where $A$ is a complemented subspace of $C_0(\alpha)$ and 
$B$ is a complemented subspace of $C_0(\alpha)\oplus C_0(\omega^\omega)$. 
\end{theorem}

In the particular case $\alpha=\omega^\omega$, from Theorem \ref{main3} and Propositions \ref{geom1} and \ref{geom3}, we may deduce that $X$ is a infinite dimensional complemented subspace of $C_0(\omega^\omega\times L)$ if and only if $X$ is isomorphic to either $c_0$, $C_0(\omega^\omega)$, $C_0(L)^n$, $C_0(\omega\times L)$ or $C_0(\omega^\omega\times L)$.

The paper is organized as follows. In the Section \ref{div} we will establish some results concerning the space $C_0(\alpha\times L)$ where $L$ is a particular type of scattered space. In Section \ref{Oper}, we will introduce a combinatorial property, prove its existence and demonstrate Theorem \ref{main1} and Corollary \ref{main2}. Finally, in Section \ref{Geo} we will investigate the complemented subspaces of $C_0(\alpha\times L)$ when all operators are of the form $T=R_L+S$ and establish Theorem \ref{main3}.


\section{The space $C_0(\alpha\times L)$}
\label{div}

In this section we will present a number of results concerning Banach spaces $C_0(\alpha\times L)$, where 
$\alpha$ is a countable ordinal space and $L$ is an element of the class $\mathcal{S}$, defined as follows: a locally compact Hausdorff space $L$ belongs to $\mathcal{S}$ if there exist finite-to-one continuous surjection of $L$ onto $[0,\omega_1)$, that is, a continuous surjective function $\varphi:L \to [0,\omega_1)$ such that $|\varphi^{-1}[\{x\}]|<\omega$ for every $x \in L$. 

\begin{proposition}\label{scat} If $L\in \mathcal{S}$ then $L$ is a scattered space.
\end{proposition}
\begin{proof}
Let $A$ be a nonempty subset of $L$. Then $\varphi[A]$ is a nonempty subset of $[0,\omega_1)$. Since $[0,\omega_1)$ is scattered, there is $x\in A$ such that $\varphi(x)$ is isolated in $\varphi[A]$, i.e., $\{\varphi(x)\}$ is an open set in $\varphi[A]$. Then $\varphi^{-1}[\{\varphi(x)\}]$ is an open set in $A$. But it is a finite open set and since $L$ is Hausdorff, $x$ is isolated in $A$. We deduce that 
$L$ is a scattered space.

\end{proof}

From now on, we will always consider for each $L\in \mathcal{S}$ a fixed finite-to-one continuous surjection $\varphi:L\to [0,\omega_1)$ together with the collection
$\{L_{\lambda}:\lambda<\omega_1\}$, where $L_\lambda=\varphi^{-1}[[0,\lambda]]$ for each $\lambda<\omega_1$. It is a collection of clopen subsets of $L$ such that $L_{\xi}\subseteq L_{\rho}$ whenever $\xi<\rho$ and 
\[L=\bigcup_{\lambda<\omega_1}L_{\lambda}.\]

From this clopen cover we deduce that each compact subset of $L$ is countable and $L$ is first-countable. Moreover, each countable subset of $K=L \cup \{\infty\}$ has countable closure.

Throughout the paper, for  each $0\leq \rho<\omega_1$ we use the following identifications: 
\begin{align*}
C_0(L_{\rho})&\cong\{g \in C_0(L): g(y)=0\text{ for each } y \in L\setminus L_{\rho}\}\\
C_0(L\setminus L_{\rho})&\cong\{g \in C_0(L): g(y)=0\text{ for each } y \in L_{\rho}\}         
\end{align*}
For $0\leq \rho<\omega_1$ we define:
\begin{align*}
A_{\rho}(\alpha)&=\{f \in C_0(\alpha\times L):f(\eta,y)=0\text{ for all }(\eta,y)\in \alpha\times (L\setminus L_{\rho})\},\\
B_{\rho}(\alpha)&=\{f \in C_0(\alpha\times L):f(\eta,y)=0\text{ for all }(\eta,y)\in \alpha\times L_{\rho}\}
\end{align*}

It is evident that for each $\rho<\omega_1$, 
\[A_\rho(\alpha)\cong C_0(\alpha)\widehat{\otimes}_{\varepsilon} C_0(L_{\rho})\text{ and }B_\rho(\alpha)\cong C_0(\alpha)\widehat{\otimes}_{\varepsilon} C_0(L\setminus L_{\rho}).\]
Since $\alpha\times L$ is the disjoint union of the clopen sets $\alpha\times L_{\rho}$ 
and $\alpha\times (L\setminus L_{\rho})$ we have \[C_0(\alpha\times L)=A_\rho(\alpha)\oplus B_\rho(\alpha).\]



In what follows we prove a number of propositions that will play some role in the process of establishing our main results.
We will always consider fixed countable ordinals $\alpha$, $\beta$; a space $L\in \mathcal{S}$ and an operator $T:C_0(\alpha\times L)\to C_0(\beta\times L)$. We denote by $T^*:C_0(\beta\times L)^*\to C_0(\alpha\times L)^*$ the dual operator of $T$.

\begin{proposition}\label{countablesupport}
If $f\in C_0(\alpha\times L)$, then $f$ has countable support.
\end{proposition}
\begin{proof}
For each $n \geq 1$, the set $\Lambda_n=\{z\in \alpha\times L: |f(z)|\geq \frac{1}{n}\}$ is compact and then countable because $L\in \mathcal{S}$. The support of $f$ is the union $\bigcup_{n=1}^{\infty}\Lambda_n$, therefore, it is countable. 
\end{proof}

\begin{proposition}\label{oper7}
If $X$ is a separable subspace of $C_0(\alpha\times L)$, then $X\subset A_{\rho}(\alpha)$ for some $\rho<\omega_1$.
\end{proposition}
\begin{proof}
Let $D=\{g_n:n \in \mathbb{N}\}$ be a countable dense subset of $X$. Since each $g_n$ has countable support, there is $\rho<\omega_1$ such that $g_n(x,y)=0$ for each $(x,y) \in \beta\times (L\setminus L_{\rho})$ for every $n \in \mathbb{N}$.

Next, let $f\in X$ and $(x,y) \in \beta\times (L\setminus L_{\rho})$ arbitrary. Given $\epsilon>0$, since $D$ is dense in $X$, there is $n \in \mathbb{N}$ such that $|f(x,y)|=|f(x,y)-g_n(x,y)|\leq\|f-g_n\|\leq \epsilon$. Then $|f(x,y)|\leq \epsilon$ for each $\epsilon>0$, i.e., $f(x,y)=0$. Hence, $f\in A_{\rho}(\alpha)$.
\end{proof}

\begin{proposition}\label{geom8} If $T$ has separable image, there is $\rho<\omega_1$ such that $T[B_{\rho}(\alpha)]$ is the null subspace.
\end{proposition}
\begin{proof}
Because $T$ has separable image, by Proposition \ref{oper7}, there is $\rho<\omega_1$ such that 
$T[C_0(\alpha \times L)]\subseteq A_{\rho}(\beta)\cong C_0(\beta\times L_{\rho})$. Towards a contradiction, assume that for each $\lambda<\omega_1$ there is $f_{\lambda} \in B_{\lambda}(\alpha)$ such that $T(f_\lambda)\neq 0$. Then, for each $\lambda<\omega_1$ there is $(x_{\lambda},y_{\lambda}) \in \beta\times L_{\rho}$ such that $T(f_\lambda)(x_{\lambda},y_{\lambda})\neq 0$. 
Since $\beta\times L_{\rho}$ is countable, without loss of generality we may assume that for some $(x,y)\in \beta\times L_{\rho}$, 
$T(f_\lambda)(x,y)\neq 0$ for each $\lambda<\omega_1$. Since  $T^*(\delta_{(x,y)})$ is atomic and has countable support there is 
$\xi<\omega_1$ such that 
\[\alpha\times (L\setminus L_{\xi})\cap \mathrm{supp}(T^*(\delta_{(x,y)}))=\emptyset.\]

Since $f_\xi(z)=0$ for each 
$z \in  \alpha\times L_{\xi}$, 
\begin{align*}T(f_\xi)(x,y)&=\int f_\xi dT^*(\delta_{(x,y)})=\int_{\alpha\times L_{\xi}} f_\xi dT^*(\delta_{(x,y)})=0,
\end{align*}
and this is a contradiction. 
\end{proof}

\begin{proposition}\label{oper1}For any countable set $A_0\subseteq L$  there is a countable set 
$B_0\subset L$ such that $\mathrm{supp}(T^*(\delta_{z}))\cap (\alpha\times A_0)=\emptyset$ 
whenever $z\in \beta\times (L\setminus B_0)$.
\end{proposition}
\begin{proof}
For otherwise, there is a countable set $A_0$ such that 
\[B=\{z\in \alpha \times L:\mathrm{supp}(T^*(\delta_{z}))\cap (\alpha \times A_0)\neq\emptyset\}\] is uncountable. 
Then, there is $w_0\in A$ and $\epsilon>0$ such that $|T^*(\delta_{z})|(\{w_0\})\geq \epsilon$ for each $z$ belonging to an uncountable subset $B_0\subseteq B$. Since $\alpha\times L$ is first countable, we may fix a countable clopen basis $\mathcal{V}_{w_0}$ for $w_0$. For each $z \in B_0$, let $V_{z}\in \mathcal{V}_{w_0}$ such that $|T^*(\delta_{z})|(V_z\setminus \{w_0\})\leq \frac{\epsilon}{2}$. 
Because $B_0$ is uncountable and $\mathcal{V}_{w_0}$ is countable, without loss of generality, we may assume that 
for some $V \in \mathcal{V}_{w_0}$, $|T^*(\delta_{z})|(V\setminus \{w_0\})\leq \frac{\epsilon}{2}$ for each $z \in B_0$.
Then,
\begin{align*}
|T(\chi_V)(z)|&=|\int \chi_V dT^*(\delta_{z})|=|T^*(\delta_{z})(V)|\\
               &\geq |T^*(\delta_{z})|(\{w_0\})-|T^*(\delta_{z})|(V\setminus \{w_0\})\geq \frac{\epsilon}{2}
\end{align*}
and this is a contradiction because, according to Proposition \ref{countablesupport}, $T(\chi_V)$ has countable support.
\end{proof}

\begin{proposition}\label{oper6}
For each operator $R:C_0(\alpha)\to C_0(\beta)$, the formula  
\[R_L(f)(x,y)=R(f\restriction_{[0,\alpha]\times \{y\}})(x)\] defines an operator $R_L:C_0(\alpha\times L)\to C_0(\beta\times L)$ with $\|R_L\|=\|R\|$.   
\end{proposition}
\begin{proof}
We first check that $R_L$ is well defined by proving that for each $f \in C_0(\alpha\times L)$, the function 
$R_L(f):[0,\beta]\times K\to \mathbb{R}$ given by the formula above is an element of $C_0(\beta\times L)$. It is clear that $R_L(f)$ vanishes at $\partial([0,\beta]\times K)$. To prove that $R_L(f)$ is continuous we pick an arbitrary point $(x,y)\in [0,\beta]\times K$  and distinguish two cases: 
\begin{case}[1]$y \neq \infty$. 
\end{case}

Let $\{(x_n,y_n)\}_{n}$ be a sequence in $\alpha\times L$ such that $(x_n,y_n)\to (x,y)$. Since $x_n \to x$ and $R(f\restriction_{[0,\alpha]\times \{y\}})\in C_0(\beta)$, there is $N_0$ such that whenever $n \geq N_0$
\[\left|R(f\restriction_{[0,\alpha]\times \{y\}})(x_n)-R(f\restriction_{[0,\alpha]\times \{y\}})(x)\right|<\frac{\epsilon}{2}.\]

Since $y_n \to y$ and $\alpha$ is a countable ordinal, $f\restriction_{[0,\alpha]\times \{y_n\}}\to f\restriction_{[0,\alpha]\times \{y\}}$ in $C_0(\alpha)$ and there is $N_1 \in \mathbb{N}$ such that whenever $n\geq N_1$
\[\|f\restriction_{[0,\alpha]\times \{y_n\}}-f\restriction_{[0,\alpha]\times \{y\}}\|<\frac{\epsilon}{2(\|R\|+1)}.\]

If $n \geq \max\{N_0,N_1\}$, then  
\begin{align*}
|R_L(f)(x_n,y_n)-R_L(f)(x,y)|&\leq |R_L(f)(x_n,y_n)-R_L(f)(x_n,y)|\\
                                 &+|R_L(f)(x_n,y)-R_L(f)(x,y)|\leq \\
                                 &\leq \|R\|\|f\restriction_{[0,\alpha]\times \{y_n\}}-f\restriction_{[0,\alpha]\times \{y\}}\|\\
                                 &+|R(f\restriction_{[0,\alpha]\times \{y\}})(x_n)-R(f\restriction_{[0,\alpha]\times \{y\}})(x)|\\
                                 &<\frac{\epsilon}{2}+\frac{\epsilon}{2}=\epsilon. 
\end{align*}
Because $L$ is first-countable and $[0,\beta]$ is metrizable we deduce that $R_L(f)$ is continuous in $(x,y)$.

\begin{case}[2]$y=\infty$
\end{case}
Let $\{(x_\gamma,y_{\gamma})\}_{\varGamma}$ be a net in $[0,\alpha]\times K$ converging to the point $(x,\infty)$ and towards a contradiction suppose $\epsilon>0$ such that $|R_L(f)(x_\gamma,y_\gamma)|\geq \epsilon$ for each $\gamma\in \varGamma$.
For each $\gamma\in \varGamma$, let $a_\gamma\in [0,\alpha]$ be such that $f(a_{\gamma},y_\gamma)=\sup_{a\in [0,\alpha]}|f(a,y_{\gamma})|$. Then
\[|R_L(f)(x_\gamma,y_\gamma)|=|R(f\restriction_{[0,\alpha]\times \{y_\gamma\}})(x_\gamma)|\leq \|R\|\sup_{x\in [0,\alpha]}|f(x,y_{\gamma})|\leq |f(a_{\gamma},y_{\gamma})|\]

Since $[0,\alpha]$ is compact, by taking a subnet if necessary we may assume that $(a_{\gamma},y_\gamma)\to (a,\infty)$. Since 
$f$ is continuous, $f(a_{\gamma},y_{\gamma})\to 0$ which is a contradiction. Therefore, $R_L(f)$ is continuous at $(x,y)$.

We deduce that $R_L:C_0(\alpha\times L)\to C_0(\beta\times L)$ is a well defined function and it is standard to check that $R_L$ is also linear and bounded with $\|R_L\|\leq \|R\|$. To see that $\|R\|\leq \|R_L\|$ we fix $\epsilon>0$ and a function $g\in C_0(\alpha)$ such that $\|g\|\leq 1$ and $\|R\|< \|R(g)\|+\epsilon$. 
Then we fix a point $y \in L$ and a function $h\in C_0(L)$ such that $0 \leq h \leq 1=h(y)$. The function $g\otimes h$ is an element of the unit ball of $C_0(\alpha\times L)$
and 
\[\|R_L\|\geq \|R_L(g\otimes h)\|=\|R(g\otimes h\restriction_{[0,\alpha]\times \{y\}})\|=\|R(g)\|\geq \|R\|-\epsilon.\] 
Since $\epsilon$ is arbitrary, we deduce that $\|R_L\|\geq \|R\|$.
\end{proof}

\begin{remark}\label{operatorRL}For each operator $R:C_0(\alpha)\to C_0(\beta)$ we denote by 
$R_L:C_0(\alpha\times L)\to C_0(\beta\times L)$ the operator given by the formula 
\[R_L(f)(x,y)=R(f\restriction_{[0,\alpha]\times \{y\}})(x).\]
\end{remark}

\begin{proposition}\label{oper8}
For each operator $R:C_0(\alpha)\to C_0(\beta)$, $R_L:C_0(\alpha\times L)\to C_0(\beta\times L)$ has separable 
image if and only if $R$ is the null operator.
\end{proposition}
\begin{proof}
Assume that $R_L$ has separable image. According to Proposition \ref{oper7} there is $\rho<\omega_1$, such that 
$R_L[C_0(\alpha\times L)]\subset A_{\rho}(\beta)$. Then, for each $f \in C_0(\alpha\times L)$, for all $(x,y) \in \beta\times (L\setminus L_{\rho})$, $R_L(f)(x,y)=0$.

Let $y_0 \in L\setminus L_{\rho}$ be an isolated point. Given an arbitrary function $h\in C_0(\alpha)$, let $\chi_{\{y_0\}}:K\to \mathbb{R}$ be the characteristic function of $\{y_0\}$ and consider $f=h\otimes \chi_{\{y_0\}}\in C_0(\alpha\times L)$.
For each $x \in \beta$ 
\[R(h)(x)=R(h\restriction_{[0,\alpha] \times \{y_0\}})(x)=R_L(f)(x,y_0)=0.\] 
Then $R(h)=0$ and we deduce that $R$ is the null operator.
\end{proof}

\begin{proposition}\label{geom6}
Let $R:C_0(\alpha)\to C_0(\beta)$ be an operator. Then, for any $\rho<\omega_1$, 
$\overline{R_L[B_\rho(\alpha)]}=R[C_0(\alpha)]\widehat{\otimes}_{\varepsilon} C_0(L\setminus L_{\rho})$.
\end{proposition}
\begin{proof}
Let $\rho<\omega_1$ arbitrary. We first prove that $R[C_0(\alpha)]\widehat{\otimes}_{\varepsilon} C_0(L\setminus L_{\rho})\subseteq \overline{R_L[B_\rho(\alpha)]}$.
Indeed, let $G=\sum_{i=1}^{m}R(f_i)\otimes g_i$ be an element of $R[C_0(\alpha)]\otimes_{\varepsilon} C_0(L\setminus L_{\rho})$, where 
$f_i\in C_0(\alpha)$ and $g_i\in C_0(L\setminus L_{\rho})$. Then 
$F=\sum_{i=1}^{m}f_i\otimes g_i\in B_{\rho}(\alpha)$ and 
\begin{align*}
R_L(F)(x,y)&=R_L(\sum_{i=1}^{m}f_i\otimes g_i)(x,y)=R(\sum_{i=1}^{m}(f_i\otimes g_i)\restriction_{[0,\alpha]\times {y}})(x)\\
                 & =\sum_{i=1}^{m}R(f_i)(x)\cdot g_i(y)=\sum_{i=1}^{m}(R(f_i)\otimes g_i)(x,y)=G(x,y).
\end{align*}
We deduce that $R[C_0(\alpha)]\widehat{\otimes}_{\varepsilon} C_0(L\setminus L_{\rho})\subseteq \overline{R_L[B_\rho(\alpha)]}$.

To check the opposite inclusion, let $G\in R_L[B_\rho(\alpha)]$ and $F \in B_\rho(\alpha)$ such that $G=R_L(F)$. Since 
$B_\rho(\alpha)\cong C_0(\alpha)\widehat{\otimes}_{\varepsilon} C_0(L\setminus L_{\rho})$, there is a sequence $(F_n)_{n}$ converging to $F$, where $F_n$ has the form $\sum_{i=1}^{m}f_i\otimes g_i$ 
where $f_i \in C_0(\alpha)$ and $g_i \in C_0(L\setminus L_{\rho})$. Then
\begin{align*}
R_L(F_n)(x,y)&=R_L(\sum_{i=1}^{m}f_i\otimes g_i)(x,y)=R(\sum_{i=1}^{m}(f_i\otimes g_i)\restriction_{[0,\alpha]\times {y}})(x)\\
                   &=\sum_{i=1}^{m}R(f_i)(x)\cdot g_i(y)=\sum_{i=1}^{m}(R(f_i)\otimes g_i)(x,y)
\end{align*}
and we deduce that $R_L(F_n)\in R[C_0(\alpha)]\otimes_{\varepsilon} C_0(L\setminus L_{\rho})$ for each $n \in \mathbb{N}$. Thus 
$G \in R[C_0(\alpha)]\widehat{\otimes}_{\varepsilon} C_0(L\setminus L_{\rho})$. Therefore 
$\overline{R_L[B_{\rho}(\alpha)]}\subseteq R[C_0(\alpha)]\widehat{\otimes}_{\varepsilon} C_0(L\setminus L_{\rho})$.

\end{proof}

For the last two results of this section, we assume that the operator
$T:C_0(\alpha\times L)\to C_0(\beta\times L)$ has the form $T=R_L+S$
where $S:C_0(\alpha\times L)\to C_0(\beta\times L)$ is an operator with separable image and $R_L$ is induced as in Remark 
\ref{operatorRL}, by an operator $R:C_0(\alpha)\to C_0(\beta)$.

\begin{proposition}\label{oper9}
If $T$ is an isomorphism, then $R$ is an isomorphism. 
\end{proposition}
\begin{proof}
We prove first that there is $\lambda>0$, such that $\|R(f)\|\geq \lambda\|f\|$  for each 
$f \in C_0(\alpha)$. Suppose that there is 
a sequence $(f_n)_{n}$ such that $\|f_n\|=1$ and $R(f_n)\to 0$.
For each isolated point in  $y\in L$ we define the function $f_{n,y}=f_n\otimes \chi_{\{y\}}\in C_0(\alpha\times L)$.

Since $S[C_0(\alpha\times L)]$ is separable, for each $n \in \mathbb{N}$ there are $y_n$ and $y'_n$ such that 
$y_n\neq y'_n$ and $\|S(f_{n,y_n})-S(f_{n,y'_{n}}) \|<\frac{1}{n}$. We then define for each $n \in \N$, the function 
$h_n=f_{n,y_n}-f_{n,y'_{n}}$. Since $\|h_n\|=1$, 
\begin{align*}
\|T(h_n)\|= \|R_L(h_n)+S(h_n)\|\leq \|R_L(h_n)\|+\|S(h_n)\|\leq 2\|R(f_n)\|+\frac{1}{n}.
\end{align*}
Thus, $T$ is not an isomorphism.

We deduce that if $T$ is an isomorphism, then $R:C_0(\alpha)\to C_0(\beta)$ is a linear embedding. Therefore, $R[C_0(\alpha)]$ is a closed subspace of $C_0(\beta)$.

Next, let $g$ be an arbitrary function in $C_0(\beta)$. Since $S[C_0(\alpha\times L)]$ is separable, by Proposition \ref{oper7}, there 
$\rho<\omega_1$ such that $S[C_0(\alpha\times L)]\subset A_{\rho}(\beta)$. Then, for each $f \in C_0(\alpha\times L)$, for all $(\xi,y) \in \beta\times (L\setminus L_\rho)$,  $S(f)(\xi,y)=0$. 
Let $y_0\in L\setminus L_\rho$ be an isolated point of $L$ and consider the function $h=g\otimes\chi_{\{y_0\}}\in C_0(\beta\times L)$. Since $T$ is an isomorphism, 
there is $f\in C_0(\alpha\times L)$ such that $T(f)=R_L(f)+S(f)=h$. For each $\xi \in \beta$, since $S(f)(\xi,y_0)=0$, we have 
\[g(\xi)=h(\xi,y_0)=R_L(f)(\xi,y_0)+S(f)(\xi,y_0)=R(f\restriction_{[0,\alpha]\times\{y_0\}})(\xi).\]
Therefore, $R$ is surjective.
\end{proof}

\begin{proposition}\label{geom7}
If $T:C_0(\alpha\times L)\to C_0(\alpha\times L)$ is a projection, then $R:C_0(\alpha)\to C_0(\alpha)$ is a projection.
\end{proposition}
\begin{proof}
We first observe that  
\begin{align*}
T^2-T&=(R_L+S)\circ(R_L+S)-(R_L+S)\\
              &=R_L^2-R_L+ R_L\circ S+S\circ R_L+S^2-S
\end{align*}
and the operator $S'=R_L\circ S+S\circ R_L+S^2-S$ has separable image.

Suppose $f \in C_0(\alpha)$ and $\epsilon>0$ such that 
$\|R^2(f)-R(f)\|\geq \epsilon$.  For each isolated point $y \in L$ define $g_{y}=f\otimes \chi_{\{y\}}\in C_0(\alpha\times L)$.  Since 
$S'$ has separable image, there are distinct isolated points $y_0$ and $y_1$ in $L$, such that 
\[\|S'(g_{y_0})-S'(g_{y_1})\|=\|S'(g_{y_0}-g_{y_1})\|<\frac{\epsilon}{2}.\]
We fix $h=g_{y_0}-g_{y_1}$. Then, for each $(\xi,y)\in \alpha \times L$,
\begin{align*}
R_L(h)(\xi,y) = \left\{
\begin{array}{rl}
0 & \text{ if } y\notin \{y_0,y_1\},\\
R(f)(\xi) & \text{ if } y=y_0\\
-R(f)(\xi) & \text{ if } y=y_1.
\end{array}\right.
\end{align*}
Let $(\xi,y)\in \alpha \times L$ arbitrary. Since $R_L^2(h)(\xi,y)= R(R_L(h)\restriction_{[0,\alpha]\times \{y\}})(\xi)$,
\begin{align*}
R_L^2(h)(\xi,y) = \left\{
\begin{array}{rl}
0 & \text{ if } y\notin \{y_0,y_1\},\\
R^2(f)(\xi) & \text{ if } y=y_0,\\
-R^2(f)(\xi) & \text{ if } y=y_1.
\end{array}\right.
\end{align*}
We deduce that $\|R_L^2(h)-R_L(h)\|=\|R^2(f)-R(f)\|$. Then
\begin{align*}
\|T^2(h)-T(h)\|&\geq \|R_L^2(h)-R_L(h)\|-\|S'(h)\|\\
               &\geq \|R^2(f)-R(f)\|-\|S'(h)\|\geq \frac{\epsilon}{2},
\end{align*}
and $T$ is not projection. We conclude that if $T^2=T$, then $R^2=R$.

\end{proof}

\section{Few operators in $C_0(\alpha\times L)$}
\label{Oper}

Based on a construction from \cite{CK} we isolate an exotic combinatorial property for a space $L$ that will imply 
that all operators on $C_0(\alpha\times L)$ have one particular form. In order to enunciate such property, we need to enrich our terminology.  

For a locally compact Hausdorff space $L$, let $K=L\cup\{\infty\}$. We will say that points $(x_1,\ldots,x_n)\in K^n$ and $(y_1,\ldots,y_m)\in K^m$ are \emph{disjoint} if $\{x_1,\ldots,x_n\}\cap \{y_1,\ldots,y_m\}=\emptyset.$  For non-zero natural numbers $m_1$, $m_2$, $m_3$, we denote by $K(m_1,m_2,m_3)$ the set of all triples $(x_1,x_2,x_3)\in K^{m_1}\times K^{m_2}\times K^{m_3}$  such that $x_1, x_2, x_3$ are pairwise disjoint. We write $\Delta(m_1,m_2,m_3)$ for the subset of $K(m_1,m_2,m_3)$ consisting of all points $(x_1,x_2,x_3)$  such that the coordinates of each $x_i$ are equal for $i=1,2,3$, moreover, the coordinates of $x_3$ are all equal to $\infty$.

\begin{definition}\label{collapsed}
A locally compact Hausdorff space $L$ is said to be \emph{$3$-collapsed} if for any given non-zero numbers $m_1,\ m_2,\ m_3\in \N$, 
any uncountable subset of $K(m_1,m_2,m_3)$, where the points are pairwise disjoint if seen as elements of $K^{m_1+m_2+m_3}$, has a accumulation point in $\Delta(m_1,m_2,m_3)$.
\end{definition}

The next theorem states that $3$-collapsed spaces exist, at least under some extra set-theoretic assumptions. We recall that $\mathcal{S}$ is the class of all locally compact Hausdorff spaces such that there exists a finite-to-one continuous surjection of $L$ onto $[0,\omega_1)$, where the latter space is endowed with usual order topology. 

\begin{theorem}\label{compactclub} 
Under the assumption of Ostaszewski's $\clubsuit$-principle, there exists a $3$-collapsed space in $\mathcal{S}$ of height $\omega$.
\end{theorem}
\begin{proof}
Let $L$ be locally compact space constructed in \cite[Proposition 3.1]{CK} with $n = 2$, and let $K=L\cup\{\infty\}$. From that construction we know that $L \in \mathcal{S}$ and $L$ has height $\omega$. To check that
$L$ is $3$-collapsed we pick non-zero natural numbers $m_1,\ m_2,\ m_3$. For any uncountable 
subset of $K(m_1,m_2,m_3)$ consisting of pairwise disjoint points 
if seen as elements of $K^{m_1+m_2+m_3}$, a similar argument as in \cite[Proposition 3.1, Claim 3]{CK} yields an accumulation 
point in $\Delta(m_1,m_2,m_3)$.
\end{proof}

We will denote by $\mathcal{C}_3$ the class of all locally compact Hausdorff spaces from $\mathcal{S}$ that are \emph{$3$-collapsed} of height $\omega$. 

It is important also to observe that such spaces cannot be obtained in ZFC without assuming some extra set-theoretic principles. Indeed, under the Ostaszewski $\clubsuit$-principle the previous theorem states that $\mathcal{C}_3\neq \emptyset$. In particular, if $L\in \mathcal{C}_3$, then $K=L\cup \{\infty\}$ satisfies the hypotheses of \cite[Theorem 1.7]{CK}, therefore $C(K\times K)$ has no complemented copy of $c_0(\omega_1)$. On the other hand, for any $L\in \mathcal{S}$, then $C(K)$ has density $\omega_1$ and also $C(K)$ has a subspace isomorphic to $c_0(\omega_1)$. Under the assumption of \textbf{MM} (Martin's Maximum),  \cite[Corollary 4.7]{GAHA}, $C(K\times K)$ has a complemented copy of $c_0(\omega_1)$. Therefore $\mathcal{C}_3=\emptyset$.

Until the end of the section we will consider fixed countable ordinals $\alpha$, $\beta$  and a locally compact Hausdorff space $L\in \mathcal{C}_3$.  We will denote the space of all operators $T:C_0(\alpha\times L)\to C_0(\beta\times L)$ by $\mathcal{B}(\alpha,\beta)$. 

\begin{lemma}\label{oper2}
Let $\{r_{j}: j \in J\}\subset \mathbb{R}$ be a bounded set. Assume that $J$ is uncountable and for each $r \in \mathbb{R}$ and for any countable $J_0\subseteq J$ there is $j \in J\setminus J_0$ such that $r_j\neq r$. Then, there exist rational numbers $p<q$ and disjoint uncountable sets
$A,\ B\subset J$, such that whenever $a\in A$ and $b\in B$, we have $r_b<p<q<r_a$
\end{lemma}
\begin{proof}
Let $M>0$ such that $|r_j|<M/2$ for every $j\in J$ and for rational numbers $-M<p,q<M$ define 
$$\varLambda_{q}=\{a\in J:r_a>q\},\ \varGamma_{p}=\{b\in J:r_b<p\}$$  
and fix 
$$q_0=\sup\{q\in \mathbb{Q}:\varLambda_{q}\text{ is uncountable }\},\ p_0=\inf\{p\in \mathbb{Q}:\varGamma_{p}\text{ is uncountable }\}.$$

It is clear that $-M<p_0,q_0<M$. In the case that $p_0=q_0=\rho$ we may fix sequences 
$(p_n)_{n}$, $(q_n)_{n}$ in $(-M,M)\ \cap \ \mathbb{Q}$ such that $p_n< \rho < q_n$ and $q_n-p_n<1/n$ 
for all $n\in \mathbb{N}$.
Therefore, the following set is countable: 
$$J_0=\left(\bigcup_{n \in \mathbb{N}}\varLambda_{q_n}\right)\cup \left(\bigcup_{n \in \mathbb{N}}\varGamma_{p_n}\right).$$

By hypothesis, there is $j \in J\setminus J_0$ such that $r_j\neq \rho$, however, since $j\notin J_0$, $p_n<r_j<q_n$ for all 
$n\in \mathbb{N}$ and we deduce that $r_j=\rho$, a contradiction. 

If $q_0<p_0$ let $\rho\in \mathbb{Q}$ be such that $q_0<\rho<p_0$. It follows that $\varLambda_{\rho}$ is countable 
and therefore, the set $\{a\in J:r_a\leq \rho\}$ is uncountable. Given any $m\in \mathbb{N}$ such that $\rho+1/m<p_0$ we have that 
$\varGamma_{\rho+1/m}$ is uncountable which is a contradiction.

We conclude that $p_0<q_0$ and if we pick any rationals $p,q$ such that $p_0\leq p<q\leq q_0$ we have that 
$\varLambda_{q}\cap \varGamma_{p}=\emptyset$. We are done by defining $A=\varLambda_{q}$ and $B=\varGamma_{p}$. 
\end{proof}

\begin{proposition}\label{oper3}
For each $T\in \mathcal{B}(\alpha, \beta)$ there is a function 
$r:[0,\alpha]\times [0,\beta]\to \R$ and a closed and countable subset $A\subset L$ such that $T^*(\delta_{(\xi,y)})(\{(\eta,y)\})=r(\eta, \xi)$
for all $\eta\leq \alpha$, $\xi\leq \beta$ and $y\in L\setminus A$.
\end{proposition}
\begin{proof}
Given $\eta\leq \alpha$ and $\xi\leq \beta$ we consider the set 
\[\{T^*(\delta_{(\xi,y)})(\{(\eta,y)\}):y \in L\}\subset \R\] 
and towards a contradiction we assume that for each $r\in \R$ and for each countable subset $L_0\subset L$ there is $y \in L\setminus L_0$ such that 
$T^*(\delta_{(\xi,y)})(\{(\eta,y)\})\neq r$. According to the Lemma \ref{oper2}, there are rationals $p<q$ and uncountable 
subets $A,\ B\subset Y$ such that, whenever $x \in A$ and $y \in B$,
\[T^*(\delta_{(\xi,x)})(\{(\eta,x)\})<p<q<T^*(\delta_{(\xi,y)})(\{(\eta,y)\}).\]

We will construct uncountable sets $\{x_{\lambda}:\lambda<\omega_1\}$ and $\{y_{\lambda}:\lambda<\omega_1\}$ of points of $A$ and $B$ respectively and uncountable sets $\{G_{\lambda}:\lambda<\omega_1\}$ and $\{H_{\lambda}:\lambda<\omega_1\}$ of pairwise disjoint finite subsets of $L$ such that 
$G_{\lambda}\cap H_{\rho}=\emptyset$ for each $\lambda,\rho<\omega_1$ and satisfying: 
\begin{itemize}
\item[(1)] $x_{\lambda}\in G_{\lambda}$, $y_{\lambda}\in H_{\lambda}$;
\item[(2)] $|T^*(\delta_{(\xi,x_{\lambda})})|\left(\alpha\times (L\setminus G_{\lambda})\right)<\frac{q-p}{6}$;
\item[(3)] $|T^*(\delta_{(\xi,y_{\lambda})})|\left(\alpha\times (L\setminus H_{\lambda})\right)<\frac{q-p}{6}$.
\end{itemize}

In order to make the notation simpler, for each $x \in L$ we denote 
\[S_{x}=\{w\in L: (\rho,w)\in \mathrm{supp}(T^*(\delta_{(\xi,x)}))\text{ for some }\rho\in \alpha\}.\]

We proceed by induction as follows. Given $\lambda<\omega_1$ assume that we have obtained sets 
$\{x_{\rho}:\rho<\lambda\}$, $\{y_{\rho}:\rho<\lambda\}$, $\{G_{\rho}:\rho<\lambda\}$, $\{H_{\rho}:\rho<\lambda\}$, 
satisfying the requirements above and let 
$$\Omega_{\lambda}=\bigcup_{\rho<\lambda} \left(G_{\rho}\cup H_{\rho}\cup S_{x_{\rho}}\cup S_{y_{\rho}}\right).$$ 

Since $\Omega_{\lambda}$ is countable, by Proposition \ref{oper1} we may pick $x_{\lambda}\in A\setminus \Omega_{\lambda}$ such that
$$\mathrm{supp}(T^*(\delta_{(\xi,x_{\lambda})}))\cap (\alpha\times \Omega_{\lambda})= \emptyset.$$

By regularity and since Radon measures in scattered spaces are atomic, we may fix a finite set $G_{\lambda}\subset (L\setminus \Omega_{\lambda})$ containing $x_{\lambda}$ such that
\begin{align*}
|T^*(\delta_{(\xi,x_{\lambda})})|\left(\alpha\times(L\setminus G_{\lambda})\right)<\frac{q-p}{6}.
\end{align*} 

Consider the countable set $\Omega'_{\lambda}=\Omega_{\lambda}\cup G_{\lambda}\cup S_{x_{\lambda}}$.  According to 
Proposition \ref{oper1} we may pick $y_{\lambda} \in B\setminus \Omega_{\lambda}$ such that 
$$\mathrm{supp}(T^*(\delta_{(\xi,y_{\lambda})}))\cap (\alpha\times \Omega'_{\lambda})=\emptyset.$$

Once more, by regularity and since Radon measures in scattered spaces are atomic, we pick a finite set $H_{\lambda}\subset (L\setminus \Omega'_{\lambda})$ 
containing $y_{\lambda}$ such that 
\begin{align*}
|T^*(\delta_{(\xi,y_{\lambda})})|\left(\alpha\times(L\setminus H_{\lambda})\right)<\frac{q-p}{6}.
\end{align*} 
and this completes the construction of the sequences.

Since $\alpha$ is first countable, locally compact and zero-dimensional, we may fix a countable local basis $\mathcal{V}_{\eta}$ for $\eta$ consisting of compact clopen sets. For each $\lambda<\omega_1$, by the regularity of the measures, we may
fix $V_\lambda\in \mathcal{V}_{\eta}$ such that
\begin{align*}
|T^*(\delta_{(\xi,x_{\lambda})})|((V_\lambda\setminus \{\eta\})\times \{x_{\lambda}\})&<\frac{q-p}{6},\\
|T^*(\delta_{(\xi,y_{\lambda})})|((V_\lambda\setminus \{\eta\})\times \{y_{\lambda}\})&<\frac{q-p}{6}.
\end{align*}
Since $\mathcal{V}_{\eta}$ is countable we may assume that $V_\lambda=V$ for each $\lambda<\omega_1$.

Next, by passing to an uncountable subset if necessary, we may assume that $|G_{\lambda}|=m_1+1$ and $|H_{\lambda}|=m_2+1$ for each $\lambda<\omega_1$ and denote $G_{\lambda}=\{x_{\lambda},a^{\lambda}_1\ldots,a^{\lambda}_{m_1}\}$, $H_{\lambda}=\{y_{\lambda},b^{\lambda}_{1}\ldots,b^{\lambda}_{m_2}\}$.

By the construction of the sets $G_{\lambda}$s and $H_{\lambda}$s we can form the collection
$$Z=\{((x_{\lambda},y_{\lambda}),(a^{\lambda}_1\ldots,a^{\lambda}_{m_1}),(b^{\lambda}_{1}\ldots,b^{\lambda}_{m_2})):\lambda<\omega_1\}$$
which is a subset of $K(2,m_1,m_2)$. We observe that the points of $Z$ are by construction pairwise disjoint when seen as points of $K^{2+m_1+m_2}$. Since $K=L\cup\{\infty\}$ is $3$-collapsed, this set has a cluster point
\[u=((z,z),(w,\ldots,w),(\infty,\ldots,\infty))\in\Delta(2,m_1,m_2).\]

Let $\{((x_{\lambda_i},y_{\lambda_i}),(a^{\lambda_{i}}_{1},\ldots,a^{\lambda_{i}}_{m_1}),(b^{\lambda_{i}}_{1},\ldots,b^{\lambda_{i}}_{m_2}))\}_{i \in I}$ be a net in $Z$ converging to $u$. We fix a clopen neighborhood $U$ of $z$ such that $U\cap \{z,w,\infty\}=\{z\}$ and by passing to a subnet if necessary, we may assume 
that $U\cap G_{\lambda_{i}}=\{x_{\lambda_{i}}\}$, $U\cap H_{\lambda_{i}}=\{y_{\lambda_{i}}\}$, for all $i \in I$.

Recalling the clopen neighborhood of $\eta$ fixed above, $V$, we may write for each $i \in I$: 
\begin{align*}
V \times U\setminus \{(\eta,x_{\lambda_{i}})\}&\subseteq \beta \times (L \setminus G_{\lambda_i})\cup ((V\setminus \{\eta\}) \times \{x_{\lambda_i}\}),\\ 
V \times U\setminus \{(\eta,y_{\lambda_{i}})\}&\subseteq \beta \times (L \setminus H_{\lambda_i})\cup ((V\setminus \{\eta\}) \times \{y_{\lambda_i}\}).
\end{align*}
We have
\begin{align*}
|T(\chi_{V\times U})(\xi,x_{\lambda_{i}})|&=|T^*(\delta_{(\xi,x_{\lambda_{i}})})(V\times U)|
\leq |T^*(\delta_{(\xi,x_{\lambda_{i}})})(\{(\eta,x_{\lambda_{i}})\})|\\
&+|T^*(\delta_{(\xi,x_{\lambda_{i}})})(V\times U\setminus \{(\eta,x_{\lambda_{i}})\})|\\
&<p+|T^*(\delta_{(\xi,x_{\lambda_{i}})})|(\beta \times (L \setminus G_{\lambda_i}))\\
&+|T^*(\delta_{(\xi,x_{\lambda_{i}})})|((V\setminus \{\eta\}) \times \{x_{\lambda_i}\})\\
&<p+\frac{q-p}{6}+\frac{q-p}{6}=\frac{2p+q}{3}.
\end{align*}

\begin{align*}
|T(\chi_{V\times U})(\xi,y_{\lambda_{i}})|&=|T^*(\delta_{(\xi,y_{\lambda_{i}})})(V\times U)|
\geq |T^*(\delta_{(\xi,y_{\lambda_{i}})})(\{(\eta,y_{\lambda_{i}})\})|\\ 
&-|T^*(\delta_{(\xi,y_{\lambda_{i}})})(V\times U\setminus \{(\eta,y_{\lambda_{i}})\})|\\
&>q- |T^*(\delta_{(\xi,y_{\lambda_{i}})})|(\beta \times (L \setminus H_{\lambda_i}))\\
&-|T^*(\delta_{(\xi,y_{\lambda_{i}})})|((V\setminus \{\eta\}) \times \{y_{\lambda_i}\})\\
&>q-\frac{q-p}{6}-\frac{q-p}{6}=\frac{2q+p}{3}
\end{align*}
Since both nets $\{x_{\lambda_{i}}\}_{i \in I}$ and $\{y_{\lambda_{i}}\}_{i \in I}$ converge to $z$, the continuity of the function $T(\chi_{V\times U})$ and the above inequalities imply 
\begin{align*}
T(\chi_{V\times U})(\xi,z)&=\lim_{i\to \infty}T(\chi_{V\times U})(\xi,x_{\lambda_{i}})\leq\frac{2p+q}{3}\\
&<\frac{2q+p}{3}\leq \lim_{i\to \infty}T(\chi_{V\times U})(\xi,y_{\lambda_{i}})=T(\chi_{V\times U})(\xi,z),
\end{align*}
which is a contradiction. We deduce that for each $\eta\leq \alpha$ and $\xi\leq \beta$ there is $r(\eta, \xi)\in \R$ and 
a countable subset $A_{\eta,\xi}\subset L$ such that $T^*(\delta_{(\xi,y)})(\{(\eta,y)\})=r(\eta, \xi)$
for all $y\in L\setminus A_{\eta,\xi}$. We are done by fixing the function $(\eta,\xi)\mapsto r(\eta,\xi)$ and the closed and countable set
\[A=\overline{\bigcup_{\substack{\eta\leq \alpha\\ \xi\leq \beta}}A_{\eta,\xi}}.\]
\end{proof}

\begin{proposition}\label{oper4} 
For any $T\in \mathcal{B}(\alpha, \beta)$ there is a closed and countable subset $B\subset L$ such that 
$T^*(\delta_{(\xi,y)})(\{(\eta,x)\})=0$
for all $\eta\leq \alpha$, $\xi\leq \beta$ and $x,y\in L\setminus B$  such that $x \neq y$.
\end{proposition}
\begin{proof}
Given $\eta\in\alpha$ and $\xi\in \beta$, towards a contradiction let us assume that for every countable subset $L_0\subset L$, 
there are $y,x\in L\setminus L_0$ such that $x \neq y$ and $T^*(\delta_{(\xi,y)})(\{(\eta,x)\})\neq 0$. We may then obtain an uncountable set $\{(x_{\lambda},y_{\lambda}):\lambda<\omega_1\}$ of pairwise disjoint points of $L^2\setminus\{(z,z):z\in L\}$ such that $|T^*(\delta_{(\xi,y_{\lambda})})(\{(\eta,x_{\lambda})\})|\neq 0$ for each $\lambda<\omega_1$.

By passing to an uncountable subset of indices if necessary we may assume that there is $\epsilon>0$ such that 
$|T^*(\delta_{(\xi,y_{\lambda})})(\{(\eta,x_{\lambda})\})|\geq\epsilon$ for all $\lambda<\omega_1$.

Since Radon measures in scattered spaces are atomic, for each $\lambda<\omega_1$ we may fix a finite set $G_{\lambda}\subseteq L$ such that 
$x_{\lambda},y_{\lambda}\in G_{\lambda}$ and $$|T^*(\delta_{(\xi,y_{\lambda})})|(\alpha\times (L\setminus G_{\lambda}))<\frac{\epsilon}{4}.$$ 

By applying the $\Delta$-system Lemma we may assume that $\{G_{\lambda}:\lambda<\omega_1\}$ constitutes a $\Delta$-system 
with root $\Delta$ and According to Proposition \ref{oper1}, $(\alpha\times \Delta)\cap \mathrm{supp}(T^*(\delta_{(\xi,y_{\lambda})}))\neq \emptyset$ at most for countably many $\lambda$s. Therefore, without loss of generality, we may assume that $\Delta=\emptyset$.

Since $\alpha$ is first countable, locally compact and zero-dimensional, we may fix a countable local basis $\mathcal{V}_{\eta}$ for $\eta$ consisting of compact clopen sets. By regularity, for each $\lambda<\omega_1$ we may fix $V_{\lambda}\in \mathcal{V}_{\eta}$ such that 
$$|T^*(\delta_{(\xi,y_{\lambda})})|((V_{\lambda}\setminus \{\eta\})\times \{x_{\lambda}\})<\frac{\epsilon}{4}$$
and because $\mathcal{V}_{\eta}$ is countable we may assume that $V_{\lambda}=V$ for all $\lambda<\omega$. 

By passing to a further uncountable subset if necessary we may assume $|G_{\lambda}|=m+2$ for all $\lambda<\omega_1$. We denote 
$G_{\lambda}=\{x_{\lambda},y_{\lambda},a^{\lambda}_1\ldots,a^{\lambda}_m\}$ and form
$$W=\{((x_{\lambda}),(a^{\lambda}_1\ldots,a^{\lambda}_m),(y_{\lambda})):\lambda<\omega_1\}$$ 
which is an uncountable set consisting of pairwise disjoint points of $K(1,m,1)$. Since $L$ is $3$-collapsed, $W$ admits an accumulation point \[u=(z,w,\ldots,w,\infty)\in \Delta(1,m,1).\] 

Let $\{((x_{\lambda_i}),(a^{\lambda_i}_1,\ldots,a^{\lambda_i}_{m}),(y_{\lambda_i}))\}_{i \in I}$ be a net in $W$ converging to $u$ and let $U$ be a clopen neighbourhood of $z$ such that $U\cap\{z,w,\infty\}=\{z\}$. By passing to a subnet if necessary, we may 
assume that $U\cap G_{\lambda_i}=\{x_{\lambda_i}\}$ for each $i \in I$.

Recalling the clopen neighborhood $V$ of $\eta$ fixed before, we may write for each $i \in I$
$$V \times U\setminus \{(\eta,x_{\lambda_{i}})\}\subseteq \beta \times (L \setminus G_{\lambda_i})\cup ((V\setminus \{\eta\}) \times \{x_{\lambda_i}\}).$$

We then have:
\begin{align*}
|T(\chi_{V\times U})(\xi,y_{\lambda_i})|&=|T^*(\delta_{(\xi,y_{\lambda_i})})(V \times U)|\geq |T^*(\delta_{(\xi,y_{\lambda_i})})(\{(\eta,x_{\lambda_i})\})|\\
&-|T^*(\delta_{(\xi,y_{\lambda_i})})|(V\times U \setminus \{(\eta,x_{\lambda_i})\})\\
&\geq |T^*(\delta_{(\xi,y_{\lambda_i})})(\{(\eta,x_{\lambda_i})\})|-|T^*(\delta_{(\xi,y_{\lambda_i})})|(\alpha\times(L\setminus G_{\lambda_i}))\\
&-|T^*(\delta_{(\xi,y_{\lambda_i})})|((V\setminus \{\eta\}) \times \{x_{\lambda_i}\})\\
&>\epsilon-\frac{\epsilon}{4}-\frac{\epsilon}{4}=\frac{\epsilon}{2}.
\end{align*}
Since $(y_{\lambda_i})_{i \in I}$ converges to $\infty$, from the previous relation we deduce that 
$$\lim_{i\to \infty}|T(\chi_{V\times U})(\xi,y_{\lambda_i})|=|T(\chi_{U\times V})(\xi,\infty)|\geq \frac{\epsilon}{2}$$ 
and this is a contradiction because $T(\chi_{V\times U})(\xi,\infty)=0$.

We deduce that for each $\eta\in\alpha$ and $\xi\in \beta$, there is a countable subset $B_{\eta,\xi}\subset L$, 
such that, for each $x,y\in B_{\eta,\xi}$, with $x \neq y$ and $T^*(\delta_{(\xi,y)})(\{(\eta,x)\})=0$. We are done by fixing the closed and countable set
\[B=\overline{\bigcup_{\substack{\eta\in \alpha\\ \xi\in\beta}}B_{\eta,\xi}}.\]
\end{proof}

We are now in position of proving two of our main results.

\begin{proof}[Proof of Theorem \ref{main1}]

From the Propositions \ref{oper3} and \ref{oper4}, there is a function $r:[0,\alpha]\times [0,\beta]\to \mathbb{R}$
and a closed and countable set $A\subset L$, such that for each $x,y\in L\setminus A$
\begin{displaymath}
T^*(\delta_{(\xi,y)})(\{(\eta,y)\})= \left\{
\begin{array}{ll}
r(\eta,\xi) & \text{ if }x=y;\\
0 & \text{ if }x\neq y.
\end{array} \right.
\end{displaymath}

%
%

Let $y_0$ in $L\setminus A$ be an isolated point. For each $h \in C_{0}(\alpha)$ and $\xi\leq \beta$ we have: 
\begin{align*}
T(h\otimes \chi_{\{y_0\}})(\xi,y_0)&=\int (h\otimes \chi_{\{y_0\}})d T^*(\delta_{(\xi,y_0)})\\
                                           &= \sum_{\eta \in \alpha} h(\eta)\cdot T^*(\delta_{(\xi,y_0)})(\{(\eta,y_0)\})= 
                                            \sum_{\eta \in \alpha} r(\eta,\xi)\cdot h(\eta).
\end{align*}
We define $R:C_0(\alpha)\to C_0(\beta)$ by setting: 
$$R(h)(\xi)=T(h\otimes\chi_{\{y_0\}})(\xi,y_0)=\sum_{\eta \in \alpha} r(\eta,\xi)\cdot h(\eta).$$ 

Because $T$ is bounded and continuous, $R$ is a well defined operator. Recalling the operator $R_L$
from Remark \ref{operatorRL} we have:
\begin{align*}
R_L(f)(\xi,y)&= R(f\restriction_{[0,\alpha]\times \{y\}})(\xi)\\
                    &=T(f\restriction_{[0,\alpha] \times \{y\}}\otimes \chi_{\{y_0\}})(\xi,y_0)= \sum_{\eta \in \alpha} 
                    r(\eta,\xi)\cdot f(\eta,y).
\end{align*}

We will prove next that the operator $S=T-R_L$ has separable image. Since $L \in \mathcal{S}$, we may fix a continuous finite-to-one surjection $\varphi:L\to [0,\omega_1)$. The collection $\{L_\lambda:\lambda<\omega_1\}$, where $L_{\lambda}=\varphi^{-1}[[0,\lambda]]$, constitutes a clopen cover for $L$ and since $A$ is countable, it is contained in some $L_{\lambda_0}$.

For each $\lambda <\omega_1$ we consider the set
\[C_{\lambda}=\{g \in C_0(\beta\times L): g(\xi,y)=0\text{ for all }(\xi,y)\in \beta\times (L\setminus L_{\lambda})\}.\]
Each $C_\lambda$ is isomorphic to $C_0(\alpha \times L_\lambda)$, therefore is separable. Towards a contradiction we assume 
that for each $\lambda<\omega_1$, $S[C_0(\alpha\times L)]\not\subset C_{\lambda}$, that is, there is 
$f_{\lambda}\in C_0(\alpha \times L)$, such that $S(f_{\lambda})\notin C_{\lambda}$. 

Then, for each $\lambda<\omega_1$, there is $(\xi_{\lambda},y_{\lambda})\in \beta\times (L\setminus L_{\lambda})$ such that $S(f_{\lambda})(\xi_{\lambda},y_{\lambda})\neq 0.$ Without loss of generality, we may assume that $y_\lambda\in L\setminus A$ for all $\lambda<\omega_1$; $y_{\lambda_1}\neq y_{\lambda_2}$ whenever $\lambda_1\neq \lambda_2$ and $\xi_{\lambda}=\xi$ for all $\lambda$. Since 
\begin{align*}S (f_{\lambda})(\xi,y_{\lambda})=\int f_{\lambda} dS^*(\delta_{(\xi,y_{\lambda})})
                                              =\sum_{\substack{\eta\in \alpha\\ x\in L}} f_{\lambda}(\eta,x)S^*(\delta_{(\xi,y_{\lambda})})(\{(\eta,x)\})\neq 0,
\end{align*} 
we may take for each $\lambda<\omega_1$, $(\eta_{\lambda},x_{\lambda})\in \alpha\times L$ such that $S^*(\delta_{(\xi,y_{\lambda})})(\{(\eta_{\lambda},x_{\lambda})\})\neq 0.$

Recalling Proposition \ref{oper1}, we may assume that $x_\lambda\in L\setminus A$ for all $\lambda<\omega_1$, $x_{\lambda_1}\neq x_{\lambda_2}$ if $\lambda_1\neq \lambda_2$ and $\eta_{\lambda}=\eta$ 
for all $\lambda$. 

If $x_\lambda\neq y_\lambda$, then  $S^*(\delta_{(\xi,y_{\lambda})})(\{(\eta,x_{\lambda})\})=0$. Therefore  
$x_\lambda=y_\lambda=y$ for each $\lambda$. But then
\begin{align*}
S^*(\delta_{(\xi,y)})(\{(\eta,y)\})&=T^*(\delta_{(\xi,y)})(\{(\eta,y)\})
-R_L^*(\delta_{(\xi,y)})(\{(\eta,y)\})\\  
&=r(\eta,\xi)-r(\eta,\xi)=0
\end{align*}
which is a contradiction. We deduce that there is $\lambda<\omega_1$, such that  $S[C_0(\alpha\times L)]\subset C_{\lambda}$, therefore, S has separable image.

To establish the uniqueness of the decomposition, assume that there is an operator $R':C_0(\alpha)\to C_0(\beta)$ and an operator $S':C_0(\alpha\times L)\to C_0(\beta\times L)$ with separable image such that
$T=R'_L+S'$. 

Then $R'_L- R_L=(R'-R)_L=S'-S$ has separable image. By Proposition \ref{oper8}, $R'=R$. Therefore $S'=S$.
\end{proof}

\begin{proof}[Proof of Theorem \ref{main2}]
If $C(M_1)\sim C(M_2)$, since $C(M_1\times K)$ is isometric to $C(M_1)\widehat{\otimes}_{\varepsilon}C(K)$ and $C(M_2\times K)$ is isometric to $C(M_2)\widehat{\otimes}_{\varepsilon}C(K)$, 
it follows that $C(M_1\times K)\sim C(M_2\times K)$. On the other hand, assume that $C(M_1\times K)\sim C(M_2\times K)$. 
If $M_1$ is uncountable, according to Miljutin theorem, $C(M_1)$ is isomorphic to $C([0,1])$ then  $C([0,1]\times K)\sim C(M_2\times K)$. Since $[0,1]\times K$ is not a scattered compact space, then $M_2\times K$ is also not scattered, see \cite[Theorem 1.5]{Candido}. We deduce that $M_2$ is an uncountable metric compacta, see \cite[Proposition 8.6.10]{Se}. 
By Mijutin theorem, $C(M_2)\sim C([0,1])$ and we are done.

Assume now that $M_1$ and $M_2$ are countable. Then, according to a result of Mazurkiewicz and Sierpi\'nski, 
see \cite{MS} or \cite[Teorema 8.6.10]{Se}, there are countable ordinals $\alpha$ and $\beta$ such that
$C(M_1)\sim C([0,\alpha])$ and $C(M_2)\sim C([0,\beta])$. Since $L$ is scattered, we deduce that 
$$C_0(\alpha\times L)\sim C([0,\alpha]\times K)\sim C([0,\beta]\times K) \sim C_0(\beta\times L).$$

Let $T:C_0(\alpha\times L)\to C_0(\beta\times L)$ be an isomorphism. By Theorem \ref{main1}, there is an operator 
$R:C_0(\alpha)\to C_0(\beta)$ and an operator $S:C_0(\alpha\times L)\to C_0(\beta\times L)$ of separable image, such that $T=R_L+S$. 
Moreover, by Theorem \ref{oper9}, $R$ is an isomorphism. Therefore, $$C(M_1)\sim C_0(\alpha)\sim C_0(\beta)\sim C(M_2).$$
\end{proof}


\section{The geometry of $C_0(\alpha\times L)$}
\label{Geo}

In this section, we will investigate the geometry of the space $C_0(\alpha\times L)$ when $\alpha$ is a countable ordinal and $L$ is an element of the class $\mathcal{S}$, defined in  Section \ref{div}, such that all operators $T:C_0(\alpha\times L)\to C_0(\alpha\times L)$ can be decomposed as 
\begin{equation}\label{decomp}T=R_L+S
\end{equation}
where $R_L$ is an operator induced by a operator $R:C_0(\alpha)\to C_0(\alpha)$ as in Remark \ref{operatorRL}, and $S:C_0(\alpha\times L)\to C_0(\alpha\times L)$ is an operator with separable range.

In the first result of this section we show that each operator of the form (\ref{decomp}) collapses 
any copy of $c_0(\omega_1)$ into a separable subspace.

\begin{proposition}\label{geom111}
Let $T:C_0(\alpha\times L)\to C_0(\alpha\times L)$ be an operator and let $Y\subset C_0(\alpha\times L)$ be a subspace isomorphic to $c_0(\omega_1)$. If $T[Y]$ is non-separable, then $T$ that cannot be decomposed as in (\ref{decomp}).
\end{proposition}
\begin{proof}
Towards a contradiction, suppose that there is an operator $R:C_0(\alpha)\to C_0(\alpha)$ and an operator with separable image $S:C_0(\alpha\times L)\to C_0(\alpha\times L)$ such that $T=R_L+S$. Suppose that there is a subspace $Y\subset C_0(\alpha\times L)$ such that 
$Y\sim c_0(\omega_1)$ and $T[Y]$ is non-separable. Let $\mathcal{C}\subset Y$ be a family equivalent to the usual unit vector 
basis of $c_0(\omega_1)$, that is, $|\mathcal{C}|=\omega_1$ and there is $M>0$ such that,
for all finite subsets $F\subset \mathcal{C}$, for all scalars $(a_g)_{g\in F}$,
\[M\max_{g\in F}\{|a_g|\}\leq \|\sum_{g\in F}a_g g\|\leq \max_{g\in F}\{|a_g|\}.\]

Because $T[Y]$ is non-separable, we may assume that $T(g)\neq 0$ for each $g\in \mathcal{C}$, furthermore, by passing to an a uncountable subset if necessary, we may suppose that there is $\delta>0$ such that $\|T(g)\|\geq \delta$ for each $g\in \mathcal{C}$. On the other hand, since $S[C_0(\alpha\times L)]$ has separable image, we may assume without lost of generality that $S(g)=0$ for each $g\in \mathcal{C}$. For otherwise, there is $\delta>0$ such that $\|S(g)\|\geq \delta$ for uncountably many $g\in \mathcal{C}$. According to a result of H. Rosenthal, \cite[Remark following Theorem 3.4]{Rose1}, $S[C_0(\alpha\times L)]$ has a subspace isomorphic to $c_0(\omega_1)$, a contradiction.

Next, we fix $e_0$ any element of $\mathcal{C}$ and assume that for an arbitrary ordinal $\xi<\omega_1$ we have obtained 
$\{e_\eta:\eta<\xi\}\subset \mathcal{C}$ such that $\{T(e_\eta):\eta<\xi\}$ is a collection of functions with pairwise disjoint supports.
 
From Proposition \ref{countablesupport} we know that the set $\bigcup_{\eta<\xi} \mathrm{supp}(T(e_\eta))$ is countable. There is $g\in \mathcal{C}$ such that \[\mathrm{supp}(T(g))\cap \left(\bigcup_{\eta<\xi} \mathrm{supp}(T(e_\eta))\right)=\emptyset.\]
For otherwise, there is $x\in \bigcup_{\eta<\xi} \mathrm{supp}(T(e_\eta))$, $\epsilon>0$ and an uncountable
subset $\{g_\gamma:\gamma<\omega_1\}\subset \mathcal{C}$ such that $|T(g_\gamma)(x)|\geq \epsilon$ for each $\gamma<\omega_1$. 
Then, for each finite subset $F\subset \omega_1$,
\begin{align*}
\epsilon\cdot |F|&\leq |\sum_{\gamma\in F} \mathrm{sign}(T(g_\gamma)(x))\cdot T(g_\gamma)(x)|=|T(\sum_{\gamma\in F} \mathrm{sign}(T(g_\gamma)(x))\cdot g_\gamma)(x)|\\
&\leq \|T(\sum_{\gamma\in F} \mathrm{sign}(T(g_\gamma)(x))\cdot g_\gamma)\|\leq \|T\|\|\sum_{\gamma\in F} \mathrm{sign}(T(g_\gamma)(x))\cdot g_\gamma\|\leq \|T\|
\end{align*}
which is a contradiction. We then fix $e_{\xi}=g$ and obtain this way recursively a family $\{e_\xi:\xi<\omega_1\}\subset \mathcal{C}$ such that 
$\{T(e_\xi):\xi<\omega_1\}$ is a collection consisting of pairwise disjointed supported functions.

Now, for each $\xi<\omega$ let $(x_\xi,y_\xi)\in \alpha\times L$ such that 
$|T(e_\xi)(x_\xi,y_\xi)|=\|T(e_\xi)\|$. Then, for each $\xi<\omega_1$, 
\begin{align*}
\delta &\leq \|T(e_\xi)\|= |T(e_\xi)(x_\xi,y_\xi)|=|R_L(e_\xi)(x_\xi,y_\xi)|\\
&=|R(e_\xi\restriction_{\alpha\times \{y_\xi\}})(x_\xi)|\leq\|R(e_\xi\restriction_{\alpha\times \{y_\xi\}})\|
\end{align*}

Because $C_0(\alpha)$ is separable, there are $\xi,\eta<\omega_1$ such that 
\[\|R(e_\xi\restriction_{\alpha\times \{y_\xi\}})-R(e_\eta\restriction_{\alpha\times \{y_\eta\}})\|<\delta.\]
And because $T(e_\xi)$ and $T(e_\eta)$ have disjoint supports
\begin{align*}
\delta &\leq |T(e_\xi)(x_\xi,y_\xi)|=|T(e_\xi)(x_\xi,y_\xi)-T(e_\eta)(x_\xi,y_\eta)|\\
&=|R(e_\xi\restriction_{\alpha\times \{y_\xi\}})(x_\xi)-R(e_\eta\restriction_{\alpha\times \{y_\eta\}})(x_\xi)|\\
&\leq\|R(e_\xi\restriction_{\alpha\times \{y_\xi\}})-R(e_\eta\restriction_{\alpha\times \{y_\eta\}})\|<\delta,
\end{align*}
a contradiction that establishes the proposition. We conclude that no such operator can be of the form $R_L+S$.
\end{proof}

\begin{remark}\label{geom2222} 
Since $|L\setminus L^{(1)}|=\omega_1$, $C_0(\alpha\times L)$ has many isometric copies of $c_0(\omega_1)$. It follows from Proposition \ref{geom111} that an operator of the form $R_L+S$ collapses each of theses copies into separable subspaces. In particular, if $C_0(\alpha\times L)$ has a complemented subspace $Y\sim c_0(\omega_1)$, the projection of $P:C_0(\alpha\times L) \to Y$ is not as in (\ref{decomp}). We deduce, for example, that there are uncountably many operators on $C_0(\alpha\times \omega_1)$ that are not as in (\ref{decomp}).
\end{remark}

From now on we will always consider $\alpha$ as a fixed countable ordinal space and $L$ 
as a fixed element of $\mathcal{S}$ such that every operator $T:C_0(\alpha\times L)\to C_0(\alpha\times L)$ is 
of the form $T=R_L+S$. We will assume all the notational conventions presented in the beginning of the Section \ref{div}.

\begin{proposition}\label{geom222}
For each $\rho<\omega_1$, $L\setminus L_\rho$ has height $\omega$.
\end{proposition}
\begin{proof}
For otherwise, $L\setminus L_\rho$ has finite height for some $\rho<\omega_1$. Recalling the properties of the family $\{L_\lambda:\lambda<\omega_1\}$, there are $\rho\leq\lambda$ and $n<\omega$ such that $\left(L\setminus L_{\lambda}\right)^{(n+1)}=\emptyset$ and $\left(L\setminus L_{\lambda}\right)^{(n)}$ is uncountable. Moreover, we may fix an uncountable subset $\{z_\xi:\xi<\omega_1\}\subset \left(L\setminus L_{\lambda}\right)^{(n)}$ and an uncountable collection $\{A_\xi:\xi<\omega\}$ constituted of clopen compact subsets of $L\setminus L_{\lambda}$ such that $z_\xi\in A_\xi$ for each $\xi<\omega_1$. 

Let $S:C_0(L\setminus L_{\lambda})\to c_0(\omega_1)$ be given by $S(f)=(f(z_\xi))_{\xi<\omega_1}$. For each 
$\epsilon>0$, because $\{x\in L\setminus L_{\lambda}: |f(x)|\geq \epsilon\}$ is compact and $\{z_\xi:\xi<\omega_1\}$ is discrete, 
the set $\{\xi<\omega_1:|f(z_{\xi})|\geq \epsilon\}$ is finite. We deduce that $S$ is a well defined operator.

On the other hand, we consider the operator $T:c_0(\omega_1)\to C_0(L\setminus L_{\lambda})$ given by the formula
\[T((a_{\xi})_{\xi<\omega_1})(x)=\sum_{\xi<\omega_1}a_{\xi}\cdot\chi_{A_\xi}(x)\]

For each $(a_{\xi})_{\xi<\omega_1}\in c_0(\omega_1)$, 
\begin{align*}
S\circ T((a_{\xi})_{\xi<\omega_1})=\left(T((a_{\xi})_{\xi<\omega_1})(z_\xi)\right)_{\xi<\omega_1}=\left(a_{\xi}\cdot\chi_{A_\xi}(z_\xi)\right)_{\xi<\omega_1}=(a_{\xi})_{\xi<\omega_1}
\end{align*}

Then, $S\circ T=I$ is the identity operator on $c_0(\omega_1)$ and it follows that $P=T\circ S$ is a projection of 
$C_0(L\setminus L_\rho)$ onto a subspace isomorphic to $c_0(\omega_1)$. Since $C_0(L\setminus L_\rho)$ is a complemented subspace of $C_0(\alpha\times L)$ we deduce that the latter space has a complemented subspace isomorphic to $c_0(\omega_1)$. According to Proposition \ref{geom111}, this is a contradiction to our assumption that $T=R_L+S$.

\end{proof}

\begin{proposition}\label{geom1}
For each $\rho<\omega_1$, if $L_{\rho}$ is infinite, then either $C_0(L_{\rho})\sim C_0(\omega)$ or $C_0(L_{\rho})\sim C_0(\omega^{\omega})$. Moreover, there is 
$\rho<\omega_1$ such that $C_0(L_{\rho})\sim C_0(\omega^{\omega})$.
\end{proposition}
\begin{proof}
Let $\rho<\omega_1$ such that $L_{\rho}$ is infinite. Since $K_\rho=L_{\rho}\cup \{\infty\}$ is a countable compactum, by a classical result of Mazurkie\-wicz-Sierpi\'nski \cite[Theorem 8.6.10]{Se}, $K_{\rho}$ is homeomorphic to an ordinal space 
$[0,\omega^\gamma n]$, where $n <\omega$ and $\gamma<\omega_1$. Since $L$ has height $\omega$, the subspace $L_{\rho}$ has height at most $\omega$. Thus, $\gamma\leq \omega$ and we deduce that  $C_0(L_{\rho})$ is isomorphic to a complemented subspace of $C_0(\omega^{\omega})$. Then, by \cite[Corollary 5.10]{Rose}, $C_0(L_{\rho})$ is either isomorphic to $C_0(\omega)$ or $C_0(\omega^{\omega})$.

Because  $L$ has height $\omega$, for each $n<\omega$ there is $a_n\in  L^{(n)}$. Since $L=\bigcup_{\lambda<\omega_1} L_\lambda$ and
$L_{\lambda_1} \subset L_{\lambda_2}$ whenever $\lambda_1\leq \lambda_2$, there is $\rho<\omega_1$ such that 
$\{a_n:n<\omega\}\subset L_{\rho}$. Since $L_{\rho}$ is a open set,  $a_n\in L_{\rho}^{(n)}$ for each $n<\omega$, and therefore $L_{\rho}$ has height $\omega$. From the first part of the proof we deduce that $C_0(L_{\rho})\sim C_0(\omega^{\omega})$
\end{proof}

\begin{proposition}\label{geom3}For each $\rho<\omega_1$, $C_0(L\setminus L_{\rho})$ is isomorphic to $C_0(L)$. 
\end{proposition}
\begin{proof}
Let $\rho<\omega_1$ be an arbitrary ordinal. According to Proposition \ref{geom222}, $L\setminus L_\rho$ has height $\omega$. By mimicking the second part of Proposition \ref{geom1} we may fix $\rho<\lambda$ such that $C_0(L_{\lambda}\setminus L_{\rho})\sim C_0(\omega^\omega)$.  Since, by the first part of Proposition \ref{geom1}, $C_0(L_{\rho})$ isomorphic to a complemented subspace of $C_0(\omega^\omega)$ we have
\begin{align*}
C_0(L)\sim & C_0(L_{\rho})\oplus C_0(L_{\lambda}\setminus L_{\rho})\oplus C_0(L\setminus L_{\lambda})\sim  C_0(L_{\rho})\oplus C_0(\omega^\omega)\oplus C_0(L\setminus L_{\lambda})\\
      \sim & C_0(\omega^\omega)\oplus C_0(L\setminus L_{\lambda})\sim C_0(L_{\lambda}\setminus L_{\rho})\oplus C_0(L\setminus L_{\lambda}) \sim C_0(L\setminus L_{\rho}).
\end{align*}
\end{proof}

\begin{proposition}\label{geom5}
Every separable complemented subspace of $C_0(\alpha\times L)$ is isomorphic 
to a complemented subspace of $C_0(\alpha)\oplus C_0(\omega^{\omega})$.
\end{proposition}
\begin{proof}

Let $X$ be a separable complemented subspace of $C_0(\alpha\times L)$. By applying the Proposition \ref{oper7} we may deduce that $X$ is complemented in $A_{\rho}(\alpha)$ for some $\rho<\omega_1$. By the second part of the proof of Proposition \ref{geom1} we may assume that $C_0(L_\rho)\sim c_0(\omega^\omega)$. The Proposition will be established by proving that either $A_{\rho}(\alpha)\sim C_0(\omega^\omega)$ or $A_{\rho}(\alpha)\sim C_0(\alpha)$.

 Recalling that $A_{\rho}(\alpha)\cong C_0(\alpha)\widehat{\otimes}_{\varepsilon} C_0(L_\rho)$, we have
\[A_{\rho}(\alpha)\sim C_0(\alpha)\widehat{\otimes}_{\varepsilon} C_0(\omega^\omega)\sim C_0(\alpha\times \omega^\omega).\]

If $\alpha<\omega$, then $A_{\rho}(\alpha)\sim C_0(\omega^\omega)$. If $\omega\leq\alpha<\omega_1$, we fix $\gamma=\sup\{\eta:\omega^{\omega^\eta}\leq \alpha\}$ and observe that $\omega^{\omega^\gamma}\leq \alpha<\omega^{\omega^{\gamma+1}}$. According to the Bessaga and Pe{\l}czy\'nski classification theorem, see \cite{BP}, we have $C_0(\alpha)\sim C_0(\omega^{\omega^\gamma})$. Consequently
\[A_{\rho}(\alpha)\sim C_0(\omega^{\omega^\gamma}\times \omega^\omega).\]
According to \cite[Lemma 2.4]{Alspash}, the height of $[0,\omega^{\omega^{\gamma}}] \times [0,\omega^{\omega}]$ is $\omega^{\gamma}+\omega+1$. Then, by Mazurkiewicz-Sierpi\'nski theorem \cite[Theorem 8.6.10]{Se}, $[0,\omega^{\omega^{\gamma}}] \times [0,\omega^\omega]$ is homeomorphic to the ordinal space $[0,\omega^{\omega^{\gamma}+\omega}p]$ for some $0<p<\omega$. Therefore, 
$C_0(\omega^{\omega^{\gamma}}\times \omega^{\omega})\sim C_0(\omega^{\omega^{\gamma}+\omega}p)\sim C_0(\omega^{\omega^{\gamma}+\omega})$.

If $\alpha<\omega^\omega$, then $\gamma=0$ and $A_{\rho}(\alpha)\sim C_0(\omega^{\omega})$. 
If $\alpha\geq \omega^\omega$, then $\gamma>0$ and $\omega^{\omega^{\gamma}}\leq \omega^{\omega^{\gamma}+\omega} < 
\omega^{\omega^{\gamma+1}}$. According to Bessaga and Pe{\l}czy\'nski classification theorem we have 
$A_{\rho}(\alpha)\sim C_0(\omega^{\omega^{\gamma}})\sim C_0(\alpha)$.

\end{proof}

\begin{proposition}\label{geom9}Let $P:C_0(\alpha\times L)\to C_0(\alpha\times L)$ be an projection. Then, there is a projection $R:C_0(\alpha)\to C_0(\alpha)$ and $\rho<\omega_1$ such that $P[C_0(\alpha\times L)]=\overline{R_L[B_\rho(\alpha)]}\oplus\overline{P[A_\rho(\alpha)]}$   
\end{proposition}
\begin{proof}
According to Proposition \ref{main1} there is a unique operator $R:C_0(\alpha)\to C_0(\alpha)$ and a unique operator with separable image $S:C_0(\alpha\times L)\to C_0(\alpha\times L)$ such that $P=R_L+S$. According to Proposition \ref{geom7}, $R:C_0(\alpha)\to C_0(\alpha)$ is also projection.

By applying Proposition \ref{geom8} we may pick $\rho \in L$ such that $S[C_0(\alpha\times L)]\subseteq A_\rho(\alpha)$ and
$S[B_\rho(\alpha)]$ is the null subspace.

Since $P[C_0(\alpha\times L)]$ is closed, it follows that 
\begin{align*}P[C_0(\alpha\times L)]=& P[A_\rho(\alpha)\oplus B_\rho(\alpha)]\\
                                    =& R_L[B_\rho(\alpha)]+P[A_\rho(\alpha)]=\overline{R_L[B_\rho(\alpha)]}+
																		\overline{P[A_\rho(\alpha)]}
 \end{align*}
Now we check that $\overline{R_L[B_\rho(\alpha)]}\cap\overline{P[A_\rho(\alpha)]}$ is the null space.
If $g \in \overline{R_L[B_\rho(\alpha)]}$, there is a sequence $(f_n)_{n}$ in $B_\rho(\alpha)$ such that 
$R_L(f_n)$ converges to $g$. If $x \in \alpha$ and $y\in L_{\rho}$,
\begin{align*}
g(x,y)=\lim_{n \to \infty}R_L(f_n)(x,y)=\lim_{n \to \infty}R(f_n\restriction_{[0,\alpha]\times \{y\}})(x)=0.
\end{align*}
If $g \in \overline{P[A_\rho(\alpha)]}$, there is a sequence $(h_n)_{n}$ in $A_\rho(\alpha)$ such that 
$P(h_n)$ converges to $g$. If $x \in \alpha$ and $y\in L\setminus L_{\rho}$,
\begin{align*}
g(x,y)=\lim_{n \to \infty}P(h_n)(x,y)=\lim_{n \to \infty}\left(R(h_n\restriction_{[0,\alpha]\times \{y\}})(x)+ S(h_n)(x,y)\right)=0.
\end{align*}
Therefore, if $g\in \overline{R_L[B_\rho(\alpha)]}\cap\overline{P[A_\rho(\alpha)]}$, then $g =0$.

\end{proof}

\begin{proof}[Proof of Theorems \ref{main3}]
By Proposition \ref{geom1}, we may pick $\rho<\omega_1$ such that $C_0(L_{\rho})$ is isomorphic to $C_0(\omega^\omega)$. Moreover, by Proposition \ref{geom3}, 
\[C_0(L)\sim C_0(L_{\rho})\oplus C_0(L\setminus L_{\rho})\sim C_0(\omega^\omega)\oplus C_0(L).\]

We deduce that 
\[C_0(\alpha\times L)\sim C_0(\alpha\times L)\oplus C_0(\alpha\times\omega^\omega).\] 

If $A$ is a complemented subspace of $C_0(\alpha)$ and $B$ is a complemented subspace of $C_0(\alpha)\oplus C_0(\omega^\omega)$, it is evident from the previous relation that $(A\widehat{\otimes}_{\varepsilon}C(K))\oplus B$ is isomorphic to a complemented subspace $C_0(\alpha\times L)$.

On the other hand, let $X$ be an infinite dimensional complemented subspace of $C_0(\alpha\times L)$ and let $P: C_0(\alpha\times L)\to C_0(\alpha\times L)$ be a projection such that $P[C_0(\alpha\times L)]=X$. 
According to Proposition \ref{geom9}, there is a projection $R:C_0(\alpha)\to C_0(\alpha)$ and $\rho<\omega_1$ such that $X=P[C_0(\alpha\times L)]=\overline{R_L[B_\rho(\alpha)]}\oplus \overline{P[A_\rho(\alpha)]}$.

It follows that $\overline{P[A_\rho(\alpha)]}$ is a separable complemented subspace of $C_0(\alpha\times L)$. Hence, by Proposition \ref{geom5}, $\overline{P[A_\rho(\alpha)]}$ is isomorphic to a complemented subspace of $C_0(\alpha)\oplus C_0(\omega^{\omega})$. From Propositions \ref{geom6} and \ref{geom3} we deduce 
$\overline{R_L[B_\rho(\alpha)]}\sim R[C_0(\alpha)]\widehat{\otimes}_{\varepsilon} C_0(L)$. 

Therefore, $X\sim (A \widehat{\otimes}_{\varepsilon} C_0(L)) \oplus B$ where $A$ is a complemented subspace of $C_0(\alpha)$ and 
$B$ is a complemented subspace of $C_0(\alpha)\oplus C_0(\omega^{\omega})$. 
\end{proof}

\section{Acknowledgements}

The author is greatly indebted with Prof. Piotr Koszmider from Institute of Mathematics of the Polish Academy of Sciences for suggestions and help. 


\bibliographystyle{amsalpha}

\end{document}